\documentclass[10pt]{article}
\usepackage[T1]{fontenc}
\usepackage{amsmath,amsfonts,amsthm,mathrsfs,amssymb}
\usepackage{cite}
\usepackage{graphicx,float}
\usepackage{subfigure}
\usepackage{placeins}
\usepackage{color}
\usepackage{indentfirst}
\usepackage[bookmarks=true]{hyperref}
\numberwithin{equation}{section}
\newcommand{\R}{{\mathbb R}}
\newcommand{\Z}{{\mathbb Z}}

\newcommand{\be}{\begin{eqnarray}}
\newcommand{\ben}{\begin{eqnarray*}}
\newcommand{\en}{\end{eqnarray}}
\newcommand{\enn}{\end{eqnarray*}}

\newcommand{\ov}{\overline}

\newcommand{\ga}{\gamma}
\newcommand{\G}{\Gamma}

\newcommand{\Om}{\Omega}
\newcommand{\om}{\omega}

\newtheorem{theorem}{Theorem}[section]
\newtheorem{lemma}[theorem]{Lemma}
\newtheorem{corollary}[theorem]{Corollary}
\newtheorem{definition}[theorem]{Definition}
\newtheorem{remark}[theorem]{Remark}

\begin{document}
\renewcommand{\theequation}{\arabic{section}.\arabic{equation}}
\begin{titlepage}
\title{\bf Analysis of the Fourier series  Dirichlet-to-Neumann boundary condition of the Helmholtz equation and its application to finite element methods}
\date{}

\author{Liwei Xu\thanks{School of Mathematical Sciences, University of Electronic Science and Technology of China, Sichuan, 611731, China. Email:{\tt xul@uestc.edu.cn}}\;, Tao Yin\thanks{Department of Computing \& Mathematical Sciences, California Institute of Technology, 1200 East California Blvd., CA 91125, United States. Email:{\tt taoyin89@caltech.edu}}}

\end{titlepage}
\maketitle
\vspace{.2in}
\begin{abstract}
 It is well known that the Fourier series Dirichlet-to-Neumann (DtN) boundary condition can be used to solve the Helmholtz equation in unbounded domains. In this work, applying such DtN boundary condition and using the finite element method (FEM), we solve and analyze a two dimensional transmission problem describing  elastic waves inside a bounded and closed elastic obstacle and acoustic waves outside it.  We are mainly interested in analyzing the DtN boundary condition of the Helmholtz equation in order to establish  the  well-posedness results of the approximated variational equation, and further derive a priori error estimates involving  effects of both the finite element discretization and the DtN boundary condition truncation. Finally,  some numerical results are presented to illustrate the accuracy of the numerical scheme.
\vspace{.2in}

\end{abstract}

\section{Introduction}
\label{sec:1}
As considering numerical solutions of an exterior problem of time-harmonic waves, one usually employs a non-reflecting  boundary condition, also known as an absorbing boundary condition,  on an artificial boundary which decomposes the exterior  region. The Fourier series Dirichlet-to-Neumann (DtN) boundary condition(\cite{F83,F84,KG89}), also called the DtN map,  is regarded as one of absorbing boundary conditions. The DtN map  actually establishes a relationship between  the Dirichlet data and the Neumann data on the artificial boundary, and it allows us to reduce the original exterior boundary value problem to a  nonlocal boundary value problem inside the artificial boundary  for which domain methods are able to be applied for numerical solutions.  In particular, if the finite element method (FEM) is employed, it leads us to the  DtN-FEM (\cite{KG89,G99}).

There exist some issues as the DtN map is employed  for the solution of an exterior boundary value problem of time-harmonic scattering waves. The first issue is the restriction  of the shape of the artificial boundary, and this problem  has been investigated  in earlier works (\cite{NN04,NN06}) in which the authors apply a  perturbation method  to the DtN map so that it can be defined on a perturbed curve of the artificial boundary with regular shapes such as a circle or an ellipse in two dimensions. The new DtN map leads to an improvement on efficiency by a reduction of  computational area between the boundary of scatters and the artificial boundary. The second one is a potential loss of the uniqueness (\cite{GK95,NN06}) of the  solution due to a truncation of the infinite series of the DtN map. In \cite{GK95}, the authors defined a modified form of the DtN map to circumvent this difficulty. The modified form of the DtN  map is equivalent to the original DtN map, and however is not equivalent to the DtN map after the truncation (\cite{HW13}). From the numerical point of view, the authors of \cite{HH92} suggested to make a choice of $N\ge kR$ for the sake of stability and accuracy, where $N$ is the truncation order of the infinite series, $k$ is the wave number and $R$ is the radius of the circular artificial boundary.  The third issue is related to the error estimates due to the truncation of the DtN map.   For problems of Laplace and linear elasto-statics, authors of  \cite{HW85,HW92} have derived an priori error estimates including  such effects. However, the techniques  in \cite{HW85,HW92} can not been applied directly to the acoustic and elastic wave problems since their corresponding bilinear forms in the weak formulations are usually indefinite. One certainly  could apply the Fredholm alternative theory and  the Schatz argument(\cite{Schatz74}), as to be presented in this work,  to perform analysis provided that the corresponding uniqueness results have to be established. In this sense, the consideration of the second issue would be of importance to deal with the third issue.  In addition, the existing results (\cite{HW85,HW92}) only theoretically indicate the algebraic decays in the form of $N^{-t}$ with $N$ being the truncation order of the infinite series, and $t$ being a positive constant. However, Numerical errors due to the truncation of the DtN map are actually understood (\cite{NN06}) to have an order of exponential decays in the form of $q^{-N}$ with  $q$ being a constant such that $0<q<1$. Similar analysis results for diffraction grating problems can be found in \cite{Bao95,Bao97,HRY}.

This work is mainly devoted to studying the last two issues of the DtN map mentioned above. To our knowledge, there are no exsiting literatures showing in a direct way the uniqueness of weak solutions of the nonlocal boundary value problem after the truncation of the DtN map, and presenting an error estimate indicating the feature of exponential decays due to the truncation of the DtN map.   We will give a novel procedure to prove the unique solvability (Theorem \ref{wellposedness2}) of the corresponding truncated variational equation instead of a proof by contradiction (\cite{HNPX11,GYX}). In order to study the errors due to the truncation of the infinite series DtN map, we derive a new and more apparent truncation error estimates declaring exponential attenuations between the exact DtN map and its truncated form (Theorem \ref{EstDiffSSN1}). Utilizing  these two results and classical finite element analysis, we are allowed to derive a priori error estimates (Theorem \ref{H1estimate} and \ref{H0estimate}) involving  effects of both finite element discretizations and infinite series truncations,  explicitly indicating the order of exponential decays of errors due to  truncations of the DtN map.

To the purpose of presentation, we consider a model problem of the time-harmonic fluid-solid interaction (FSI) problem (\cite{HKS88, H94,HKR00}). The model describes  the scattering of   time-harmonic acoustic incident waves by an elastic body immersed in fluids, and is of great importance in many fields such as exploration seismology, oceanography, and non-destructive testing, and to name a few.  Numerical solutions of the time-harmonic  FSI problem have been studied  widely using the FEM (\cite{BM91,NN09,DSM121,DSM122,GHM09,GMM07, GMM12,HKM08,MMS04,YRX}), or the boundary integral equation (BIE) methods (\cite{LM95,YHX}). Here, we declare  that our main concerns consist of the analysis of the DtN boundary condition as it is coupled with FEM to solve problems in unbounded domains, without intentions of performing sophisticated finite element analysis for the model problem (\cite{NN09,DSM121,DSM122,GHM09,GMM07,GMM12}), or for time-harmonic waves (\cite{MS10}).

The remainder of the paper is organized as follows. We first describe the classical FSI problem (Section \ref{sec:2.1}) and then reduce the transmission problem to an equivalent nonlocal boundary value problem (Section \ref{sec:2.2}). Essential mathematical analysis for the corresponding variational equation of the nonlocal boundary value problem and its modification due to truncation of the DtN map are discussed in Section \ref{sec:2.2} and \ref{sec:3}, respectively. In Section \ref{sec:4}, based on a point estimate of the DtN map, we establish a priori error estimates including effects of both the finite element discretization and the DtN boundary condition truncation for the finite element solution of the modified variational equation. Several numerical experiments are presented to confirm our theoretical results in Section \ref{sec:5}.

\section{Statement of the problem}
\label{sec:2}
\begin{figure}
\centering
{\includegraphics[scale=1.5]{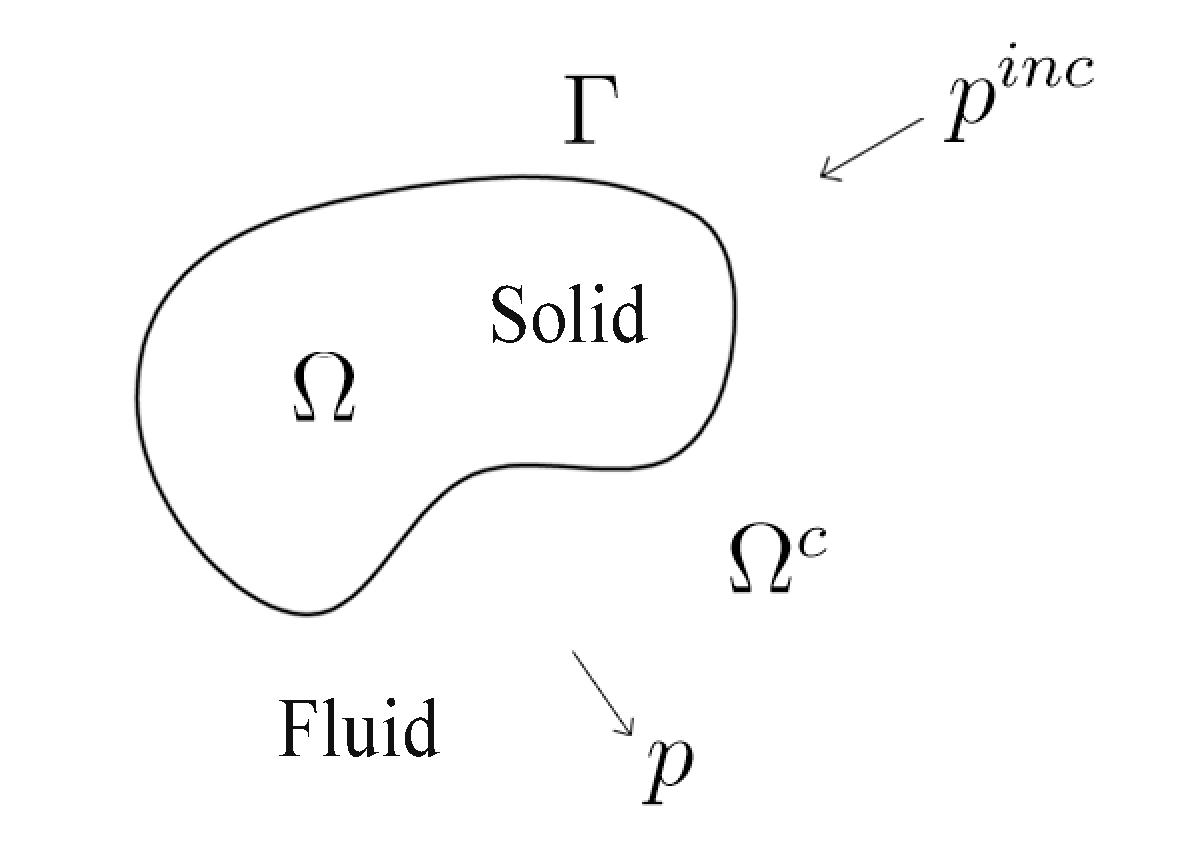}}
\caption{The geometry settings for the fluid-solid interaction problem (\ref{Navier1})-(\ref{RadiationCond}).}\label{fig1}
\end{figure}

\subsection{Original boundary value problem}
\label{sec:2.1}
Let $\Om\subset \R^2$ denote a bounded and simply connected domain with smooth interface  $\G= \partial\Om$,  and $\Om^c= \R^2\setminus\overline{\Om}\subset \R^2$ be the unbounded exterior region. The domain $\Om$ is occupied by a linear and isotropic elastic solid determined through the Lam\'e constants $\lambda$ and $\mu$ ($\mu$>0, $\lambda+\mu>0$) and its mass density $\rho>0$, and $\Om^c$ is filled with compressible, inviscid fluids. We denote by $c_0>0$ the speed of sound in the fluid, $\rho_f>0$ the density of the fluid, $\omega$ the frequency and $k=\omega/c_0$ the wave number. The fluid-solid interaction problem to be considered in this work is to determine the elastic displacement $u$ in the solid and the acoustic scattered field $p$ in the fluid provided an incident field $p^{inc}$. It can be described mathematically as follows. Given $p^{inc}$, find $u=(u_x,u_y)^\top\in \left(C^2(\Om)\cap C^1(\overline\Om)\right)^2$ and $p\in C^2(\Om^c)\cap C^1(\overline{\Om^c})$ satisfying
\be
\label{Navier1}
\Delta^{*}u +   \rho \om^2u = 0&&\mbox{in}\quad \Om,\\
\label{Helmholtz1}
\Delta p + k^2p = 0&&\mbox{in}\quad \Om^c,\\
\label{TransCond1.1}
\om^2\rho_fu\cdot \nu -\partial_\nu p-\partial_\nu p^{inc}=0 &&\mbox{on}\quad \G,\\
\label{TransCond2.1}
Tu +\nu p + \nu p^{inc}=0&&\mbox{on}\quad \G,
\en
and the Sommerfeld radiation condition
\be
\label{RadiationCond}
\lim_{r \to \infty} r^{\frac{1}{2}}\left(\frac{\partial p}{\partial r}-ikp\right) =  0,\quad r=|x|,
\en
uniformly with respect to all $\hat{x}=x/|x|\in\mathbb{S}:=\{x\in\R^2:|x|=1\}$. Here, $\partial_\nu$ is the normal derivative on $\G$ (here and in the sequel, $\nu$ is always the outward unit normal to the boundary), $i=\sqrt{-1}$ is the imaginary unit. $\Delta^{*}$ is the operator defined by
\ben
\Delta^* = \mu\Delta + (\lambda + \mu)\nabla\nabla\cdot,
\enn
and the traction operator $T$ is defined by
\ben
Tu = 2\mu\partial_\nu u + \lambda\nu\nabla\cdot u + \mu\nu\times\mbox{curl}\,u.
\enn

It is known (\cite{JDS83}) that the problem (\ref{Navier1})-(\ref{RadiationCond}) is not always uniquely solvable due to the occurrence of traction free oscillations for certain geometries and some frequencies $\omega$. These $\omega$ are also known as the Jones frequencies which are inherent to the original model. We conclude from \cite{LM95} the following uniqueness result.

\begin{lemma}
\label{tractionfree}
If the surface $\Gamma$ and the material parameters $(\mu,\lambda,\rho)$ are such that there are no traction free solutions, the boundary value problem (\ref{Navier1})-(\ref{RadiationCond}) has at most one solution. Here, we call a nontrivial $u_0$ a traction free solution if it solves
\ben
\label{Eq:TracEq1}
\Delta^{*}u_0 + \rho\omega^2u_0 =0 &&\mbox{in}\quad \Omega,\\
\label{Eq:TracEq2}
Tu_0 =0 &&\mbox{on}\quad\Gamma,\\
\label{Eq:TracEq3}
u_0\cdot\nu =0 &&\mbox{on}\quad\Gamma.
\enn
\end{lemma}

To simplify the presentation throughout the dissertation, we shall denote by $c>0$, $\alpha>0$ generic constants whose precise values are not required and may change line by line.

\subsection{Nonlocal boundary value problem}
\label{sec:2.2}
We introduce an artificial circular boundary $\G_R:=\{x\in\R^2:|x|=R\}$ such that $\ov{\Om}\subset B_R:=\{x\in\R^2:|x|<R\}$. The artificial boundary decomposes the exterior domain $\Om^c$ into two subdomains. One is the annular region $\Om_R=B_R\backslash\ov{\Om}$ between $\G$ and $\G_R$, and the other is the unbounded exterior region $\Om_R^c=\mathbb{R}^{2}\backslash \overline{B_R}$. On $\G_R$, we impose the exact transparent boundary condition
\be
\label{DtNBoundaryCond}
\partial_\nu p = Sp \quad \mbox{on}\quad  \G_R,
\en
where $S$ is known as the DtN map from $H^s(\G_R)$ to $H^{s-1}(\G_R)$ defined by
\be
\label{DtNmap}
S\varphi:={\sum\limits_{n=0}^{\infty}} \ '\frac{ {k H_{n}^{(1)}}^{\prime}{(kR)}}{\pi H_{n}^{(1)}(kR)}\int_ {0}^{2 \pi}\varphi(R,\phi)\cos(n(\theta-\phi))d\phi \quad\mbox{for all}\quad \varphi \in H^{s}(\G_R),1/2\leq s\in \R,
\en
or equivalently,
\be
\label{DtNmap1}
S\varphi:=\sum\limits_{n\in\Z}\frac{ {k H_{n}^{(1)}}^{\prime}{(kR)}}{2\pi H_{n}^{(1)}(kR)}\int_ {0}^{2 \pi}\varphi(R,\phi)e^{in(\theta-\phi)}d\phi \quad\mbox{for all}\quad \varphi \in H^{s}(\G_R),1/2\leq s\in \R.
\en
Here and in the sequel, the prime behind the summation means that the first term in the summation is multiplied by $1/2$. The transparent boundary condition (\ref{DtNBoundaryCond}) defines a nonlocal condition for $p$ on $\G_R$ since the Dirichlet data $p$ over the entire boundary $\G_R$ are required to compute the Neumann date $\partial_\nu p$ at a single point $x\in \G_R$. We have the following mapping property for $S$ (see Theorem 3.1 in \cite{HNPX11}).
\begin{lemma}
\label{Sbound}
For any constant $s\ge 1/2$, the DtN map $S$ defined by (\ref{DtNmap}) or (\ref{DtNmap1}) is a bounded linear operator from $H^{s}(\G_{R})$ to $H^{s-1}(\G_R)$, that is, there exists a constant $c>0$ independent of $\varphi\in H^{s}(\G_R)$ such that
\be
\label{DtNcontinuity}
\|S\varphi\|_{H^{s-1}(\G_R)} \leq c\|\varphi\|_{H^s(\G_R)}.
\en
\end{lemma}

Now, the transmission problem (\ref{Navier1})-(\ref{RadiationCond}) can be equivalently replaced by the following nonlocal boundary value problem: Given $p^{inc}$, find $u\in \left(C^2(\Om)\cap C^1(\overline\Om)\right)^2$ and $p\in C^2(\Om_R)\cap C^1(\overline{\Om_R})$ such that
\be
\label{Navier2}
\Delta^{*}u +   \rho \om^2u = 0&&\mbox{in}\quad \Om,\\
\label{Helmholtz2}
\Delta p + k^2p = 0&&\mbox{in}\quad \Om_R,\\
\label{TransCond1.2}
\om^2\rho_fu\cdot \nu -\partial_\nu p-\partial_\nu p^{inc}=0 &&\mbox{on}\quad \G,\\
\label{TransCond2.2}
Tu +\nu p + \nu p^{inc}=0&&\mbox{on}\quad \G,\\
\label{DtNCond}
\partial_\nu p-Sp=0 &&\mbox{on}\quad  \G_R.
\en
The uniqueness result for the nonlocal boundary value problem (\ref{Navier2})-(\ref{DtNCond}) is established in the next Theorem.
\begin{theorem}
\label{uniqueness}
If the surface $\G$ and the material parameter $(\mu,\lambda,\rho)$ are such that there are no traction free solutions, the nonlocal boundary value problem (\ref{Navier2})-(\ref{DtNCond}) has at most one solution.
\end{theorem}
\begin{proof}
The proof is analogous to that of Theorem 3.1 in \cite{GYX} and we omit it here.
\end{proof}

Set $\mathcal{H}^t= \left(H^{t}(\Om)\right)^2\times H^{t}(\Om_R)$. Now we study the weak formulation of (\ref{Navier2})-(\ref{DtNCond}) which reads: Given $p^{inc}$, find $U=(u,p)\in \mathcal{H}^1$ such that
\be
\label{Weak1}
A(U,V)=B(U,V)+b(p,q)=\ell(V),\quad\forall\ V=(v,q)\in\mathcal{H}^1,
\en
where
\ben
\label{B}
B(U,V)&=& a_{1}(u,v)+a_{2}(p,q)+a_{3}(u,q)+a_{4}(p,v),
\enn
and
\ben
\label{a1}
a_{1}(u,v) &=&\int_{\Om}{ \left[\lambda(\nabla\cdot u)(\nabla\cdot\ov{v})+\frac{\mu}{2}\left(\nabla u+(\nabla{u})^T\right):\left(\nabla\ov{v}+(\nabla\ov{v})^T\right)- \rho\om^{2}u \cdot \ov {v}\right]\,dx},\\
\label{a2}
a_{2}(p,q)&=&\int_{\Om_R}{ \left(\nabla p\cdot\nabla\ov{q}-k^2p\ov{q}\right)\,dx},\\
\label{a3}
a_{3}(u,q) &=& \rho_f\om^2\int_{\G}{u\cdot \nu \ov q\,ds},\\
\label{a4}
a_{4}(p,v) &=& \int_{\G}{\nu p\cdot \ov {v}\,ds},\\
\label{b}
b(p,q) &=& -\int_{\G_R}{(Sp)\ov q\,ds}
\enn
are sesquilinear forms defined on $\left(H^{1}(\Om)\right)^2\times \left(H^{1}(\Om)\right)^2$, $H^{1}(\Om_R)\times H^{1}(\Om_R)$, $\left(H^{1}(\Om)\right)^2\times H^{1}(\Om_R)$, $H^{1}(\Om_R)\times \left(H^{1}(\Om)\right)^2$ and $H^{1}(\Om_R)\times H^{1}(\Om_R)$, respectively, and $\ell$ defined by
\ben
\label{lv}
\ell(V)=\int_{\G}{\partial_\nu p^{inc}\ov{q}\,ds}-\int_{\G}{\nu\,p^{inc}\cdot\ov {v}\,ds},
\enn
is a linear functional on $\mathcal{H}^1$ dependent on $\left(p^{inc},\partial_\nu p^{inc}\right)\in H^{1/2}(\G)\times H^{-1/2}(\G).$

\begin{remark}
The double dot notation appeared in the above formulation is understood in the following way. If tensors ${\bf A}$ and ${\bf B}$ have rectangular Cartesian components $a_{ij}$ and $b_{ij}$, $i,j=1,2$, respectively, then the double contraction of ${\bf A}$ and ${\bf B}$ is
\ben
{\bf A}:{\bf B}=\sum\limits_{i = 1}^2 {\sum\limits_{j = 1}^2 {{a_{ij}}{b_{ij}}}}.
\enn
\end{remark}

It follows from the boundedness of the sesquilinear form $B(\cdot,\cdot)$ and (\ref{DtNcontinuity}) that
\begin{theorem}
\label{continuous1}
The sesquilinear form $A(\cdot,\cdot)$ in (\ref{Weak1}) satisfies
\ben
\label{Weak1boundness}
A(U,V)|\leq c\|U\|_{\mathcal{H}^1}\|V\|_{\mathcal{H}^1},\quad\forall\ U,V\in{\mathcal{H}^1},
\enn
where $c>0$ is the continuity constant independent of $U$ and $V$.
\end{theorem}

\begin{theorem}
\label{garding1}
The sesquilinear form $A(\cdot,\cdot)$ in (\ref{Weak1}) satisfies a Garding's inequality taking the form
\ben
\label{Weak1garding}
\mbox{Re}\,\{A(V,V)+(CV,V)_{\mathcal{H}^1}\}\geq \alpha\|V\|_{\mathcal{H}^1}^2,\quad\forall\,V=(v,q)\in{\mathcal{H}^1},
\enn
where $C:\mathcal{H}^1\rightarrow \mathcal{H}^1$ is a compact operator, $(\cdot,\cdot)_{\mathcal{H}^1}$ denotes the inner product on $\mathcal{H}^1$ and $\alpha>0$ is a constant independent of $V$.
\end{theorem}
\begin{proof}
For the sesquilinear form $a_{1}(\cdot,\cdot)$, we know from the Korn's inequality (\cite{H95}) that there exist constants $\alpha>0, \widetilde{\alpha}\ge 0$ such that
\be
\label{a1vv}
a_{1}(v,v)+\rho\om^{2}\left\| v \right\|_{{{\left( {{H^0}(\Om)} \right)}^2}}^2 \geq \alpha \left\| v \right\|_{{{\left( {{H^1}(\Om)} \right)}^2}}^2-\widetilde{\alpha}\left\| v \right\|_{{{\left( {{H^0}(\Om)} \right)}^2}}^2.
\en
The notation $\widetilde{\alpha}$ will be used again in Appendix \ref{appendix1}. For the sesquilinear form $a_{2}(\cdot,\cdot)$, we obtain
\ben
\label{a2qq}
a_{2}(q,q) &=& \left\|q \right\|_{H^1(\Om_R)}^2-(k^2+1)\left\|q \right\|_{H^0(\Om_R)}^2.
\enn
Next, we consider the sesquilinear form $b(\cdot,\cdot)$. Define
\ben
b_{1}(p, q)=\frac{1}{\pi}{\sum\limits_{n=1}^{\infty}}\ n \int_{0}^{2 \pi}\int_ {0}^{2 \pi}p(R,\phi)\overline{q(R,\theta)}\cos(n(\theta-\phi))d\theta d\phi
\enn
and
\ben
b_{2}(p, q)=\frac{kR}{\pi}{\sum\limits_{n=0}^{\infty}} \ '\frac{{H_{n}^{(n-1)}}{(kR)}}{H_{n}^{(1)}(kR)}\int_{0}^{2 \pi}\int_ {0}^{2 \pi}p(R,\phi)\overline{q(R,\theta)}\cos(n(\theta-\phi))d\theta d\phi
\enn
implying that
\be
\label{b=b1-b2}
b(p, q)=b_{1}(p, q)-b_{2}(p, q).
\en
It follows from Theorem 4.2 in \cite{HNPX11} that
\be
\label{b1}
b_{1}(q, q) \ge0,
\en
and
\ben
\label{b2}
\left|\mbox{Re}\,\{b_{2}(q, q)\}\right| \le |b_{2}(q, q)| \le c \left\|q \right\|_{H^{0}(\G_R)}^2.
\enn
The compact operator $C:\mathcal{H}^1\rightarrow \mathcal{H}^1$ can be defined as
\ben
(CU,V):= (\widetilde{\alpha}+\rho\omega^2)\int_{\Omega}\,u\cdot\ov{v}dx+ (k^2+1)\int_{\Omega_R}\,p\,\ov{q}dx -a_{3}(u,q)-a_{4}(p,v)+ b_{2}(p, q),\quad\forall\,V\in{\mathcal{H}^1}.
\enn
This completes the proof.
\end{proof}

Now, the existence result follows immediately from the Fredholm Alternative theorem: Uniqueness implies the existence. As a consequence of Theorem \ref{continuous1} and \ref{garding1}, we have the following theorem.

\begin{theorem}
\label{wellposedness1}
Let the surface $\G$ and the material parameter $(\mu,\lambda,\rho)$ be such that there are no traction free solutions, then the variational equation (\ref{Weak1}) admits a unique solution.
\end{theorem}

\section{Modified nonlocal boundary value problem}
\label{sec:3}
\subsection{Truncated DtN map}
\label{sec:3.1}
In practical computing, one needs to truncate the infinite series of the exact DtN map at a finite order to obtain an approximate DtN map written as
\be
\label{TruncDtNmap}
S^{N}\varphi:={\sum\limits_{n=0}^{N}}\ '\frac{ {k H_{n}^{(1)}}'{(kR)}}{\pi H_{n}^{(1)}(kR)}\int_ {0}^{2 \pi}\varphi(R,\phi)\cos(n(\theta-\phi))d\phi,
\en
or
\be
\label{TruncDtNmap1}
S^N\varphi:=\sum\limits_{n=-N}^N\frac{ {k H_{n}^{(1)}}^{\prime}{(kR)}}{2\pi H_{n}^{(1)}(kR)}\int_ {0}^{2 \pi}\varphi(R,\phi)e^{in(\theta-\phi)}d\phi,
\en
for all $\varphi \in H^{s}(\G_R), \ s\geq1/2$. Obviously, $S^N$ is also a linear bounded operator. Here, the non-negative integer $N$ is called the truncation order of the DtN map. Consequently, we arrive at a truncated nonlocal boundary value problem consisting of (\ref{Navier2})-(\ref{TransCond2.2}) and
\be
\label{TruncDtNCond}
\partial_\nu p = S^{N}p \quad  \text {on } \G_R.
\en

\subsection{Modified variational problem and wellposedness}
\label{sec:3.2}
We now consider the modified variational equation of (\ref{Weak1}) due to the truncation of the DtN map: looking for $U_N=(u_N,p_N) \in{\mathcal{H}^1}$ such that
\be
\label{Weak2}
A^N(U_N,V)=B(U_N,V)+b^N(p_N,q)=\ell(V),\quad \forall\; V=(v,q)\in{\mathcal{H}^1},
\en
where
\ben
b^N(p_N,q)=-\int_{\G_R}{(S^Np_N)\ov q ds}.
\enn

\begin{theorem}
\label{continuousgarding2}
The sesquilinear form $A^N(\cdot,\cdot)$ in (\ref{Weak2}) satisfies
\be
\label{Weak2continuity}
\quad\left|A^N(U,V)\right|\leq c\|U\|_{\mathcal{H}^1}\|V\|_{\mathcal{H}^1},\quad\forall\;U,V\in{\mathcal{H}^1},
\en
and
\be
\label{Weak2garding}
\quad\mbox{Re}\,\{A^N(V,V)+(C^NV,V)_{\mathcal{H}^1}\}\geq \alpha\|V\|_{\mathcal{H}^1}^2,\quad\forall\;V=(v,q)\in{\mathcal{H}^1},
\en
where $C^N:\mathcal{H}^1\rightarrow \mathcal{H}^1$ is a compact operator and $c>0, \alpha>0$ are constants independent of $U, V$.
\end{theorem}
\begin{proof}
It is easy to deduce (\ref{Weak2continuity}) due to the fact that $S^N$ is also bounded. We now show the proof of (\ref{Weak2garding}). Following (\ref{b=b1-b2}), we define
\ben
b_1^N(p, q)=\frac{1}{\pi}{\sum\limits_{n=1}^{N}}\ n \int_{0}^{2 \pi}\int_ {0}^{2 \pi}p(R,\phi)\ov{q(R,\theta)}\cos(n(\theta-\phi))d\theta d\phi
\enn
and
\ben
b_2^N(p, q)=\frac{kR}{\pi}{\sum\limits_{n=0}^{N}} \ '\frac{{H_{n-1}^{(1)}}{(kR)}}{H_{n}^{(1)}(kR)}\int_{0}^{2 \pi}\int_ {0}^{2 \pi}p(R,\phi)\ov{q(R,\theta)}\cos(n(\theta-\phi))d\theta d\phi
\enn
which yields
\ben
b^N(p, q)=b_1^N(p, q)-b_2^N(p, q).
\enn
It follows from Theorem 4.4 in \cite{HNPX11} that
\ben
b_1^N(q, q)\ge0,
\enn
and
\be
\label{b2N}
\left|\mbox{Re}\,\{b_2^N(q, q)\}\right| \le |b_2^N(q, q)| \le c \left\|q \right\|_{H^{0}(\G_R )}^2.
\en
Consequently, by the definition of compact operator $C$ and (\ref{b2N}), we complete the proof of (\ref{Weak2garding}).
\end{proof}

Now we present the uniqueness and existence of solutions for the modified variational equation (\ref{Weak2}) which is the key ingredient of this work and put its proof in Appendix \ref{appendix1}.

\begin{theorem}
\label{wellposedness2}
Let the surface $\Gamma$ and the material parameter $(\mu,\lambda,\rho)$ be such that there are no traction free solutions, then there exists a constant $N_0\ge0$ such that the modified variational equation (\ref{Weak2}) admits a unique solution $(u_N,p_N)\in\mathcal{H}^1$ for $N\geq N_0$.
\end{theorem}

\section{Finite element analysis}
\label{sec:4}
Our main goal in this section is to establishing a priori error estimates (\cite{HW85,HW92,HNPX11,GYX}) for the finite element solution of (\ref{Weak2}) in terms of the finite element mesh size $h$ and the truncation order $N$ in appropriate Sobolev spaces. It has been known that numerical errors induced from the truncation of DtN map are exponentially decaying (\cite{NN06}), and thanks to the error estimate \eqref{EstDiffSSN1}, we are able to derive a new upper bound of numerical errors indicating such effect explicitly.

\subsection{Point estimate for DtN map}
\label{sec:4.1}
We include the point estimate for the difference of $S$ and $S^N$ proposed in \cite{HNPX11} in the next Lemma.
\begin{lemma}
\label{Theorem4.1}
Suppose that the DtN maps $S$ and $S^{N}$ are defined as in (\ref{DtNmap}) and (\ref{TruncDtNmap}), respectively. Then, for given $\varphi \in H^{s}(\G_R),s \in \mathbb{R}$, there holds, for all $t\geq 0$,
\be
\label{EstDiffSSN}
\left\|(S-S^{N})\varphi\right\|_{H^{s-1}(\G_R)}\leq c\frac{\epsilon(N,\varphi)}{N^{t}}\|\varphi\|_{H^{s+t}(\G_R)},
\en
where $c>0$ is a constant independent of $\varphi$ and $N$, $\epsilon (N,\varphi )$ is a function of the truncation order $N$ and $\varphi$ for given values of $s$ and $t$, generated by the addition of leading terms to the summation with positive terms for the construction of the norm on the space $H^{s+t}(\G_R)$, satisfying $\epsilon(N,\varphi)\le 1$ and $\epsilon(N,\varphi)\rightarrow 0$ as $N\rightarrow \infty$ for all $\varphi\in H^{s}(\G_R)$.
\end{lemma}

From the estimation (\ref{EstDiffSSN}) we can observe that the truncation error between $S$ and $S^N$ tends to zero as $N\rightarrow\infty$. However, when $t=0$, the attenuation property depends on $\epsilon(N,\varphi)$ and we know few about this function. We now give a new and more apparent truncation error estimate for some special $\varphi\in H^{s}(\G_R)$ which will be used in the subsequent discussions.
\begin{theorem}
\label{Theorem4.2}
Suppose that the DtN maps $S$ and $S^{N}$ are defined as in (\ref{DtNmap1}) and (\ref{TruncDtNmap1}), respectively. Let $p$ be a solution of Helmholtz equation outside $\Om$ satisfying either (\ref{DtNBoundaryCond}) or (\ref{TruncDtNCond}). Then there exists a $N_0>0$ such that for all $N>N_0$,
\be
\label{EstDiffSSN1}
\left\|(S-S^N)p\right\|_{H^{s-1}(\G_R)}\leq cq^N\|p\|_{H^{s+t+1/2}(\Om_R)},\quad\forall\; t\ge 0,\;s\ge 1/2,
\en
where $0<q<1$ is a constant independent of $N$.
\end{theorem}
\begin{proof}
{\bf Case 1:} Suppose that $p$ satisfies (\ref{DtNBoundaryCond}).

We know that $p$ admits an expansion taking the form
\ben
p(x)=\sum_{n\in\Z}p_n\frac{H_n^{(1)}(k|x|)}{H_n^{(1)}(k|R|)}e^{in\theta}\quad\mbox{for all} \quad |x|>\G^+.
\enn
Here, $\G^+:=\max\{|x|, x\in\G\}$. It easily follows that
\ben
\left\|(S-S^N)p\right\|_{H^{s-1}(\G_R)} &=& \left\{\sum_{|n|>N} (1+n^2)^s|p_n|^2
\frac{k^2}{1+n^2}\left|\frac{{H_n^{(1)}}^{\prime}(kR)}{H_n^{(1)}(kR)}\right|^2\right\}^{1/2}\\
&\leq& c\left\{\sum_{|n|>N}(1+n^2)^s|p_n|^2\right\}^{1/2}\\
&\leq& c\sum_{|n|>N}(1+n^2)^{s/2}|p_n|.
\enn
For some small $\varepsilon>0$, we have
\ben
(1+n^2)^{s/2}|p_n|\left|\frac{H_n^{(1)}(k(\G^++2\varepsilon))}{H_n^{(1)}(kR)}\right|
&=& (1+n^2)^{(s+t)/2}|p_n|\left|\frac{H_n^{(1)}(k(\G^++\varepsilon))}{H_n^{(1)}(kR)}\right|\\
&\times& (1+n^2)^{-t/2}\left|\frac{H_n^{(1)}(k(\G^++2\varepsilon))}{H_n^{(1)}(k(\G^++\varepsilon))}\right|\\
&\le& \left\{\sum_{n\in\Z}(1+n^2)^{s+t}|p_n|^2\left|\frac{H_n^{(1)}(k(\G^++\varepsilon))}{H_n^{(1)}(kR)}\right|^2\right\}^{1/2}\\
&\times& \left\{\sum_{n\in\Z}(1+n^2)^{-t}\left|\frac{H_n^{(1)}(k(\G^++2\varepsilon))}{H_n^{(1)}(k(\G^++\varepsilon))}\right|^2\right\}^{1/2}.
\enn
Note that for sufficiently large $|n|$ and $z>0$ (see e.g. \cite{AS65}),
\ben
H_{|n|}^{(1)}(z)\sim -i\sqrt{\frac{2}{|n|\pi}}\left(\frac{e}{2|n|}\right)^{-|n|}z^{-|n|}.
\enn
Then we deduce that there exists a $N_0>0$ such that for $|n|>N_0$,
\ben
(1+n^2)^{-t}\left|\frac{H_n^{(1)}(k(\G^++2\varepsilon))}{H_n^{(1)}(k(\G^++\varepsilon))}\right|^2 \le c(1+n^2)^{-t}\left(\frac{\G^++\varepsilon}{\G^++2\varepsilon}\right)^n,
\enn
and
\ben
\left|\frac{H_n^{(1)}(kR)}{H_n^{(1)}(k(\G^++2\varepsilon))}\right| \le cq^{|n|},
\enn
where $q:=(\G^++2\varepsilon)/R<1$. These further imply that
\ben
\sum_{n\in\Z}(1+n^2)^{-t}\left|\frac{H_n^{(1)}(k(\G^++2\varepsilon))}{H_n^{(1)}(k(\G^++\varepsilon))}\right|^2 <\infty.
\enn
Then for $|n|>N_0$, we deduce from the trace theorem that
\ben
(1+n^2)^{s/2}|p_n|\le cq^{|n|}\|p\|_{H^{s+t}(\G_{\G^++\varepsilon})}\le cq^{|n|}\|p\|_{H^{s+t+1/2}(\Om_R)}.
\enn
We conclude that
\ben
\left\|(S-S^N)p\right\|_{H^{s-1}(\Gamma_{R})}&\le& c\sum_{|n|>N}q^{|n|}\|p\|_{H^{s+t+1/2}(\Om_R)}\\
&\le& cq^N\|p\|_{H^{s+t+1/2}(\Om_R)}.
\enn

{\bf Case 2:} Suppose that $p$ satisfies (\ref{TruncDtNCond}).

We know that $p$ admits an expansion taking the form
\ben
p(x)=\sum_{n\in\Z}\left(p_n^{(1)}\frac{H_n^{(1)}(k|x|)}{H_n^{(1)}(k|R|)}+p_n^{(2)}\frac{H_n^{(2)}(k|x|)}{H_n^{(2)}(k|R|)}\right)e^{in\theta}\quad\mbox{for all} \quad |x|>\G^+.
\enn
According to the condition (\ref{TruncDtNCond}), we conclude that
\ben
p_n^{(1)}\frac{{H_n^{(1)}}^{\prime}(kR)}{H_n^{(1)}(kR)} +p_n^{(2)}\frac{{H_n^{(2)}}^{\prime}(kR)}{H_n^{(2)}(kR)}=
\begin{cases}
(p_n^{(1)}+p_n^{(2)})\frac{{H_n^{(1)}}^{\prime}(kR)}{H_n^{(1)}(kR)},\quad &|n|\le N ,\cr 0,\quad &|n|>N,\nonumber
\end{cases}
\enn
which implies that
\ben
p_n^{(2)}=
\begin{cases}
0,\quad &|n|\le N ,\cr -p_n^{(1)}\frac{{H_n^{(1)}}^{\prime}(kR)}{H_n^{(1)}(kR)} \frac{H_n^{(2)}(kR)}{{H_n^{(2)}}^{\prime}(kR)},\quad &|n|>N.\nonumber
\end{cases}
\enn
It follows that
\ben
\left\|(S-S^N)p\right\|_{H^{s-1}(\G_R)} &=& \left\{\sum_{|n|>N} (1+n^2)^s|p_n^{(1)}+p_n^{(2)}|^2
\frac{k^2}{1+n^2}\left|\frac{{H_n^{(1)}}^{\prime}(kR)}{H_n^{(1)}(kR)}\right|^2\right\}^{1/2}\\
&\leq& c\left\{\sum_{|n|>N}(1+n^2)^s|p_n^{(1)}+p_n^{(2)}|^2\right\}^{1/2}\\
&\leq& c\sum_{|n|>N}(1+n^2)^{s/2}|p_n^{(1)}| \left|1-\frac{{H_n^{(1)}}^{\prime}(kR)}{H_n^{(1)}(kR)} \frac{H_n^{(2)}(kR)}{{H_n^{(2)}}^{\prime}(kR)}\right|.
\enn
The proof of the remaining assertions is analogous to that of Case 1.
\end{proof}

\subsection{Galerkin formulation}
\label{sec:4.2}
Let $\mathcal{H}_{h}=({S}_{h},S'_{h})\subset \mathcal{H}^1$ be the standard finite element space. Now we consider the Galerkin formulation of (\ref{Weak2}): Given $p^{inc}$, find ${U}_h=({u}_h,p_h)\in \mathcal{H}_h$, ${u}_h=(u_x^h,u_y^h)^\top$ such that
\begin{equation}
\label{Weak3}
A^N({U}_h,{V}_h)=\ell({V}_h),\quad \forall\; {V}_h=({ v}_h,q_h)\in\mathcal{H}_h,\;{v}_h=(v_x^h,v_y^h)^\top.
\end{equation}
It can be shown (\cite{HW04}) that the discrete sesquilinear form $A^N(\cdot,\cdot)$ satisfies the BBL-condition as implication of the following:

{\it G{\aa}rding's inequality + Uniqueness + Approximation property of $\mathcal{H}_{h}\Rightarrow$ BBL-condition.}

\begin{theorem}
\label{Galerkin}
Let the surface $\G$ and the material parameter $(\mu,\lambda,\rho)$ be such that there are no traction free solutions and suppose that the finite element space $\mathcal{H}_h\subset\mathcal{H}^1$ satisfies the standard approximation property, then there exist constants $N_{0}\geq 0$ and $h_{0}>0$ such that for any  $h\in(0,h_0]$ and $N\geq N_0$, $A^N(\cdot,\cdot)$ satisfies the BBL condition in the form
\be
\label{BBLcondition}
\sup_{(0,0)\neq {W}_{h}\in \mathcal{H}_{h}} \frac{|{A}^{N}({V}_{h},{ W}_{h})|}{\|{W}_{h})\|_ {\mathcal{H}^1}}\geq \gamma\|{ V}_{h}\|_{\mathcal{H}^1},\quad \forall\;{V}_{h}\in\mathcal{H}_{h}.
\en
Here $\gamma>0$ is the inf-sup constant independent of $h$.
\end{theorem}

From the BBL condition (\ref{BBLcondition}), we are allowed to derive a priori error estimates for the finite element solution ${U}_{h}\in\mathcal{H}_h$.

\subsection{Asymptotic error estimates}
\label{sec:4.3}
In this subsection, we mainly derive a reasonable priori error estimates on appropriate Sobolev spaces including error effects of both the numerical discretization and the truncation of infinite series. We first establish an upper bound of numerical errors analogous to the well-known C\'ea's lemma in the positive definite case.

\begin{theorem}
\label{Theorem5.2}
There exist constants  $h_{0}>0$ and $N_0\ge0$ such that for any  $h\in(0,h_0]$ and $N\geq N_0$
\be
\label{uppererror}
\|{U}-{U}_{h}\|_{\mathcal{H}^1}\leq c\left\{\inf_{{V}_{h}\in \mathcal{H}_{h}}\|{U}-{V}_{h}\|_ {\mathcal{H}^1}+\sup_{0\neq w_2\in {S}'_{h}} \frac{|(b(p,w_2)-b^{N}(p,w_2)|}{\|w_2\|_{H^{1}(\Om_R)}}\right\}
\en
where $c>0$ is a constant independent of $h$ and $N$.
\end{theorem}
\begin{proof}
For the proof of this lemma, we refer to \cite{HNPX11,GYX}.
\end{proof}

From the estimate (\ref{uppererror}), we can see that the numerical errors is constructed by two single terms. We study the first term correlated with the finite element meshsize $h$ by using the approximation theory, and the second term dependent on the truncation order $N$ by employing the point estimate (\ref{EstDiffSSN1}). In the following, starting with the estimate (\ref{uppererror}), we derive a priori error estimates measured in $\mathcal{H}^1$-norm and $\mathcal{H}^0$-norm, respectively to conclude this section.

\begin{theorem}
\label{H1estimate}
Suppose that ${U}\in \mathcal{H}^{t}$ for $2\leq t \in \mathbb{R}$. Then there exist constants  $h_0>0$ and $N_0\geq 0$ such that for any  $h\in(0,h_0]$ and $N\geq N_0$
\be
\label{energyest}
\|{U}-{U}_{h}\|_ {\mathcal{H}^1}\leq c \left\{ h^{t-1}\|{ U}\|_{\mathcal{H}^{t}}+q^N\|p\|_{{H}^{t}(\Omega_{R})}\right\},
\en
where $c>0$ and $0<q<1$ are constants independent of $h$ and $N$.
\end{theorem}
\begin{proof}
We have known from Theorem \ref{Theorem5.2} that
\be
\label{5-3-1}
\|{U}-{U}_{h}\|_{\mathcal{H}^1}\leq c\left\{\inf_{{V}_{h}\in \mathcal{H}_{h}}\|{U}-{V}_{h}\|_{\mathcal{H}^1}+\sup_{0\neq w_2\in {S}'_{h}} \frac{|b(p,w_2)-b^{N}(p,w_2)|}{\|w_2\|_{H^{1}(\Om_R)}}\right\},
\en
where $c>0$ is a positive constant. The approximation property of the finite element space $\mathcal{H}_{h}$ gives
\be
\label{5-3-2}
\inf_ {{V}_{h}\in \mathcal{H}_{h}}\|{U}-{V}_{h}\|_ {\mathcal{H}^1}\leq c h^{t-1}\|{U}\|_{\mathcal{H}^{t}},
\en
where $c>0$ is a positive constant independent of $h$. We now consider the second term in (\ref{5-3-1}). The trace theorem implies that there exists a bounded linear operator $\gamma:{H}^1(\Omega_{R})\rightarrow {H}^{1/2}(\G_R)$ such that
\be
\label{5-3-3}
\frac{|b(p,w_2)-b^{N}(p,w_2)|}{\|w_2\|_{H^{1}(\Om_R)}}=\frac{|\langle (S-S^{N})\gamma p,\gamma w_2\rangle_{\G_R}|}{\|w_2\|_{H^{1}(\Om_R)}}=\frac{|\langle \gamma^{*}(S-S^{N})\gamma p, w_2\rangle_{\Omega_R}|}{\|w_2\|_{H^{1}(\Om_R)}},
\en
where $\langle \cdot,\cdot\rangle_{\G_R}$ is the standard $L^{2}$ duality pairing between $H^{-1/2}(\G_R)$ and $H^{1/2}(\G_R),$ and $\gamma^{*}:H^{-1/2}(\G_R)\rightarrow ({H}^{1}(\Om_R))'$ is the
adjoint operator of $\gamma$. Therefore, we have
\be
\label{5-3-4}
\sup_{0\neq w_2\in {S}'_{h}} \frac{|b(p,w_2)-b^{N}(p,w_2)|}{\|w_2\|_{H^{1}(\Om_R)}}
&=& \sup_{0\neq w_2\in {S}'_{h}} \frac{|\langle \gamma^{*}(S-S^{N})\gamma p, w_2\rangle_{\Om_R}|}{\|w_2\|_{H^{1}(\Omega_{R})}}\nonumber\\
&=& \|\gamma^{*}(S-S^{N})\gamma p\|_{ ({H}^{1}(\Om_R))'}\nonumber\\
&\leq& cq^N\|p\|_{{H}^{t}(\Om_R)}
\en
because of Theorem \ref{Theorem4.2} and the boundedness of operators $\gamma$ and $\gamma^*$. The proof is hence established by following a
combination of (\ref{5-3-2}) and (\ref{5-3-4}).
\end{proof}

We now extend the error estimate in the energy space to the one measured in the $\mathcal{H}^0$ space.

\begin{theorem}
\label{H0estimate}
Suppose that ${U}\in \mathcal{H}^{t}$ for $2\leq t \in \mathbb{R}$. Then there exist constants  $h_{0}>0$ and $N_{0}\geq 0$ such that for any  $h\in(0,h_{0}]$ and $N\geq N_{0}$
\be
\label{L2est}
\|{U}-{U}_{h}\|_ {\mathcal H^0}\leq c \left\{ h^{t}\|{ U}\|_{\mathcal{H}^{t}}+q^N\|p\|_{{H}^{t}(\Omega_{R})}\right\}
\en
where $c>0$ and $0<q<1$ are constants independent of $h$ and $N$.
\end{theorem}
\begin{proof}
Let ${E}=({e}_{1},e_{2})={U}-{U}_h$ be the finite element error. Then we have
\be
\label{5-4-1}
B({E},{V}_h)+b^{N}(e_{2},v_2)+b(p,v_2)-b^N(p,v_2)=0,\quad\forall\ {V}_h=({ v}_1,v_2)\in\mathcal{H}_{h}.
\en
Now, we consider the following boundary value problem: Find ${W}=({ w}_1,w_2)$ satisfying
\be
\label{duality1}
\Delta^{*}{w}_1 +   \rho \omega^2{w}_1 =e_1 &&\text{in}\quad \Omega,\\
\label{duality2}
\Delta w_2 + k^2w_2 = e_2 &&\text{in}\quad {\Omega_R},\\
\label{duality3}
\rho_f\omega^2{w}_1\cdot \nu-\partial_\nu w_2=0 && \text{on}\quad \Gamma,\\
\label{duality4}
{T}{w}_1 +\nu\, w_2=0 && \text{on}\quad \Gamma,\\
\label{duality5}
\partial_\nu w_2- Sw_2=0 && \text {on}\quad  \Gamma_{R}.
\en
Let ${W}=(w_1,w_2)$ be a weak solution of nonlocal boundary value problem (\ref{duality1})-(\ref{duality5}), hence ${W}$ satisfies
\be
\label{5-4-2}
{A}^{N}({V},{W})+b(v_2,w_2)-b^{N}(v_2,w_2)=({w}_1,{ e}_1)_{\left(L^2(\Omega)\right)^2}+(w_2,e_2)_{L^2(\Omega_R)},\quad \forall\; {V}=({v}_1,v_2)\in{\mathcal{H}^1}.
\en
Replacing ${V}$ by ${E}$ in (\ref{5-4-2}) gives
\be
\label{5-4-3}
{A}^{N}({E},{W})+b(e_2,w_2)-b^{N}(e_2,w_2)=({e}_1,{ e}_1)_{\left(L^2(\Omega)\right)^2}+(e_2,e_2)_{L^2(\Omega_R)}=\|{E}\|_{\mathcal H^0}^2.
\en
Then subtracting (\ref{5-4-3}) from (\ref{5-4-1}) leads to, for ${V}_{h}\in\mathcal{H}_{h},$
\be
\label{5-4-4}
\|{E}\|_{\mathcal H^0}^2=A^N({E},{W}-{ V}_h)+b(e_{2},w_2)-b^{N}(e_{2},w_2)+b^{N}(p,v_2)-b(p,v_2).
\en
Theorem \ref{continuousgarding2}, the approximation property of $\mathcal{H}_{h}$ and the regularity theory imply that
\be
\label{5-4-5}
|\textit{A}^{N}({E},{W}-{V}_{h})|&\leq& c\|{E}\|_{\mathcal{H}^1}\|{W}-{V}_{h}\|_{\mathcal{H}^1}\nonumber\\
&\leq& ch\|{E}\|_{\mathcal{H}^1}\|{W}\|_{\mathcal{H}^{2}}\nonumber\\
&\leq& ch\|{E}\|_{\mathcal{H}^1}\|{E}\|_{\mathcal{H}^{0}},
\en
where $c>0$ is a constant. Following the same argument in Theorem \ref{H1estimate} and choosing $t = 2$, we arrive at, by the regularity theory,
\be
\label{5-4-6}
|b(e_{2},w_2)-b^{N}(e_{2},w_2)|&\leq& \|\gamma^{*}(S-S^{N})\gamma e_2\|_{ ({H}^{2}(\Omega_{R}))'}\|w_2\|_{H^{2}(\Omega_{R})} \nonumber\\
&\leq& c\|(S-S^{N})\gamma e_2\|_{H^{-3/2}(\Gamma_{R})}\|{W}\|_{\mathcal{H}^{2}}\nonumber\\
&\leq& cq^N\|e_2\|_{{H}^{1}(\Omega_{R})}\|{E}\|_{\mathcal{H}^{0}}.
\en
Similarly, we also have
\ben
|b^{N}(p,w_2-v_2)-b(p,w_2-v_2)|&\leq& \|\gamma^{*}(S-S^{N})\gamma p\|_{ ({H}^{1}(\Omega_{R}))'}\|w_2-v_2\|_{H^{1}(\Omega_{R})} \nonumber\\
&\leq& c\|(S-S^{N})\gamma p\|_{H^{-1/2}(\Gamma_{R})}\|{W}-{V}_{h}\|_{\mathcal{H}^1}\nonumber\\
&\leq& chq^N\|\gamma p\|_{{H}^{t-1/2}(\Gamma_{R})} \|{W}\|_{\mathcal{H}^{2}}\nonumber\\
&\leq& chq^N\|p\|_{{H}^{t}(\Omega_{R})}\|{E}\|_{\mathcal{H}^{0}},
\enn
and
\ben
|b(p,w_2)-b^{N}(p,w_2)|&\leq& \|\gamma^{*}(S-S^{N})\gamma p\|_{ ({H}^{2}(\Omega_{R}))'}\|w_2\|_{H^{2}(\Omega_{R})} \nonumber\\
&\leq& c\|(S-S^{N})\gamma p\|_{H^{-3/2}(\Gamma_{R})}\|{W}\|_{\mathcal{H}^{2}}\nonumber\\
&\leq& cq^N\|p\|_{{H}^{t}(\Omega_{R})}\|{E}\|_{\mathcal{H}^{0}}.
\enn
Thus, by the triangular inequality we know
\be
\label{5-4-7}
|b^{N}(p,v_2)-b(p,v_2)|&\leq& |b^{N}(p,w_2-v_2)-b(p,w_2-v_2)|+|b^{N}(p,w_2)-b(p,w_2)|\nonumber\\
&\leq& c_1hq^N\|p\|_{H^{t}(\Omega_{R})}\|{E}\|_{\mathcal{H}^{0}} +c_2q^N\|p\|_{H^{t}(\Omega_{R})}\|{E}\|_{\mathcal{H}^{0}}.
\en
Therefore, by the combination of the inequalities (\ref{5-4-4})-(\ref{5-4-7}) and (\ref{energyest}) we have
\ben
\|{E}\|_{\mathcal H^0} &=& A^N({E},{W}-{ V}_h)+b(e_{2},w_2)-b^{N}(e_{2},w_2)+b^{N}(p,v_2)-b(p,v_2) \nonumber\\
&\le& c h\left\{ h^{t-1}\|{U}\|_{\mathcal{H}^{t}}+q^N\|p\|_{{H}^{t}(\Omega_{R})}\right\}
+cq^N\left\{ h^{t-1}\|{U}\|_{\mathcal{H}^{t}} +q^N\|p\|_{{H}^{t}(\Omega_{R})}\right\}\nonumber\\
&+& chq^N\|p\|_{{H}^{t}(\Omega_{R})}
+cq^N\|p\|_{{H}^{t}(\Omega_{R})}.
\enn
Finally, according to the fact that $h\in(0,h_{0}]$ and $N\geq N_{0}$, we arrive at (\ref{L2est}).
\end{proof}

\section{Numerical experiments}
\label{sec:5}
\begin{figure}
\centering
{\includegraphics[scale=0.3]{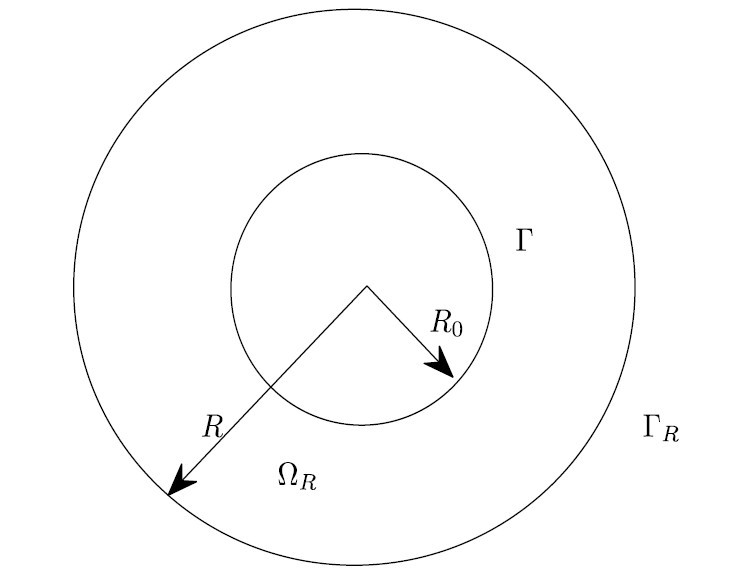}}
\caption{The computational domain of the model problem.}\label{fig2}
\end{figure}

In this section, we present several numerical tests to validate our theoretical results, and we always take the parameters $\omega=1$, $\mu=1$, $\lambda=1$, $\rho=1$, $\rho_f=1$ and $R_0=1$.  We first introduce a model problem whose analytical solution is available  so that we are able to evaluate the accuracy of the numerical solution. We consider the scattering of a plane incident wave $p^{inc}=e^{ikx\cdot d}$ with direction $d=(1,0)$ by a disc-shaped elastic body of radius $R_0$ (see Figure \ref{fig2}). For the exact solution of this model problem, we refer to \cite{HKM08,YRX}. In numerical computations, the computational domain $\Om$ and $\Om_R$ are discretized by uniform triangle elements and we employ piecewise linear basis functions $\{\varphi_i\}_{i=1}^{i=N_1}$ and $\{\psi_i\}_{i=1}^{i=N_2}$ in $\Om$ and $\Omega_R$, respectively,  to construct the finite element space $\mathcal{H}_{h}$. Here, $N_1$ and $N_2$ are the total number of elements in $\Om$ and $\Om_R$, respectively. To find the finite element solution of (\ref{Weak3}), one needs to compute the integrals
\be
\label{numbn1}
\int_{\G_R}{(S^N\psi_j)\ov{\psi_i} ds},
\en
for those $\psi_i$ and $\psi_j$  not vanishing on $\G_R$. For the sake of simplicity, (\ref{numbn1}) can be approximated by
\be
\label{numbn2}
\int_{\G_R}{(S^N\psi_j)\ov{\psi_i} ds}\approx \int_{\G_R}{(S^N\zeta_j)\ov{\zeta_i} ds},
\en
where $\{\zeta_i\}$ are the corresponding piecewise linear basis functions in terms of  $\theta$ on $\G_R$. Thus, the computation of integrals (\ref{numbn1}) amounts to evaluating a series as
\ben
\int_{\Gamma_{R}}{(S^N\psi_j)\ov{\psi_i} ds} \approx {\sum\limits_{n=-N}^N}\ \frac{ {kR H_{n}^{(1)}}^{\prime}{(kR)}}{\pi H_{n}^{(1)}(kR)}\int_ {0}^{2\pi}\zeta_j(R,\phi)e^{-in\phi}d\phi \int_ {0}^{2 \pi}\ov{\zeta_i(R,\theta)}e^{in\theta} d\theta.
\enn

In the numerical experiments, we firstly check the accuracy of the DtN-FEM. We choose $R=2$, $N=20$ and consider three different wave numbers $k=1$, 2 and 4. Figure 3 shows the exact and numerical solutions of the elastic displacement ${u}=(u_x,u_y)$ in the solid and the acoustic scattered field $p$ in the fluid for the problem (\ref{Navier1})-(\ref{RadiationCond}) with $k=1$ and meshsize $h=0.1076$. We can observe that the numerical solutions are in a perfect agreement with the exact series solutions. According to Theorem \ref{H1estimate} and \ref{H0estimate}, as the truncation order $N$ of the DtN map is large enough that the domain discretization error is dominant, we should be able to observe that
\ben
\|{U}-{U}_{h}\|_ {\mathcal H^0}=O(h^2),\quad\|{U}-{U}_{h}\|_ {{\mathcal H}^1}=O(h).
\enn
In Figure 4, we show the log-log plot of errors ${U}-{U}_{h}$ measured in $\mathcal H^0$ and $\mathcal H^1$-norms with respect to $1/h$ which verifies the convergence order of $O(h^2)$ and $O(h)$. Finally, we are concerned with the effects of truncation order $N$ on the total numerical errors. We choose $k=1$ and compute the numerical errors measured in $\mathcal H^0$-norm for three different meshsizes $h=$0.4304, 0.2151 and 0.1076, respectively. The log-log plots of errors are presented in Figure 5 showing that the errors due to the truncation of the DtN map decay extremely fast, arriving at the low bound correlating to each finite element mesh size.

\begin{figure}[ht]
\centering
\begin{tabular}{ccc}
&\includegraphics[scale=0.2]{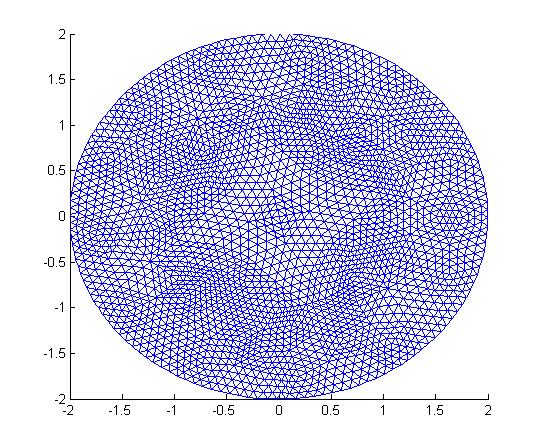}&  \\
& Mesh & \\
\includegraphics[scale=0.2]{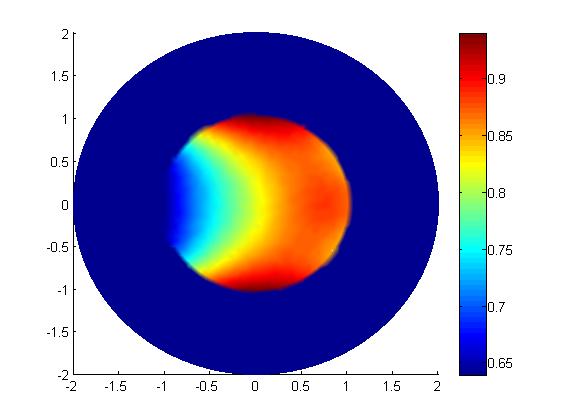} &
\includegraphics[scale=0.2]{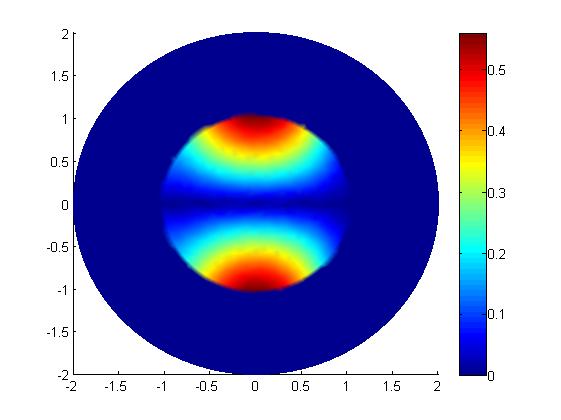} &
\includegraphics[scale=0.2]{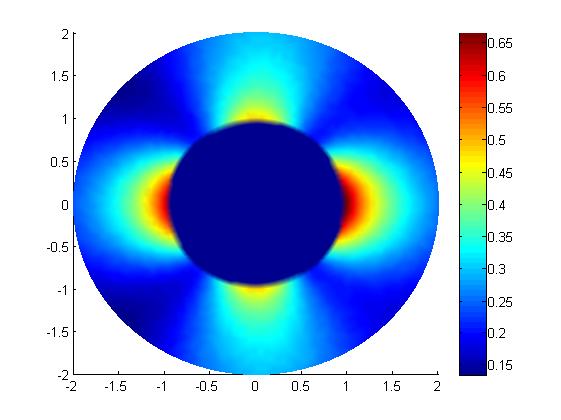} \\
(a) $|u_x|$ & (b) $|u_y|$ & (c) $|p|$ \\
\includegraphics[scale=0.2]{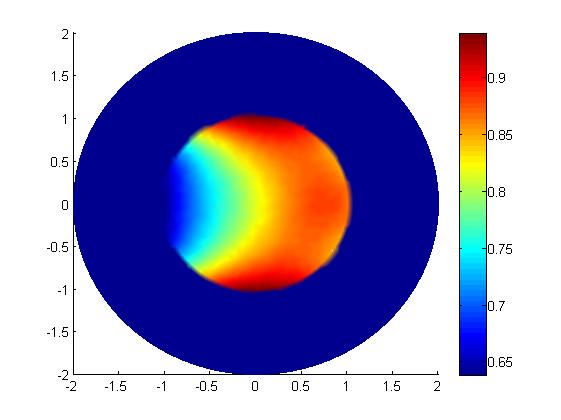} &
\includegraphics[scale=0.2]{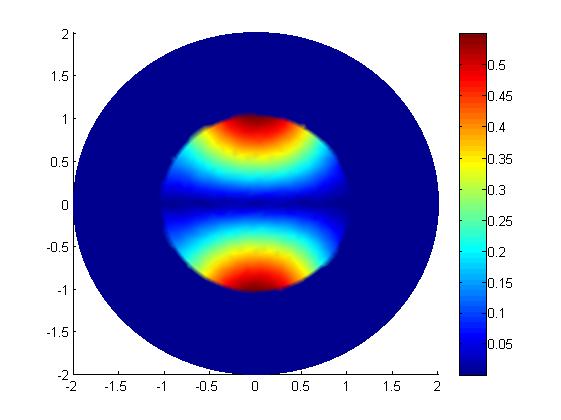} &
\includegraphics[scale=0.2]{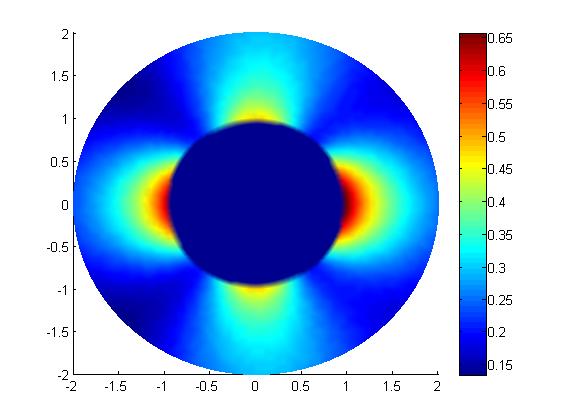} \\
(d) $|u_x|$ & (e) $|u_y|$ & (f) $|p|$ \\
\end{tabular}
\caption{Exact solutions (a,b,c) and numerical solutions of DtN-FEM (d,e,f).}
\end{figure}

\begin{figure}[ht]
\centering
\begin{tabular}{cc}
\includegraphics[scale=0.4]{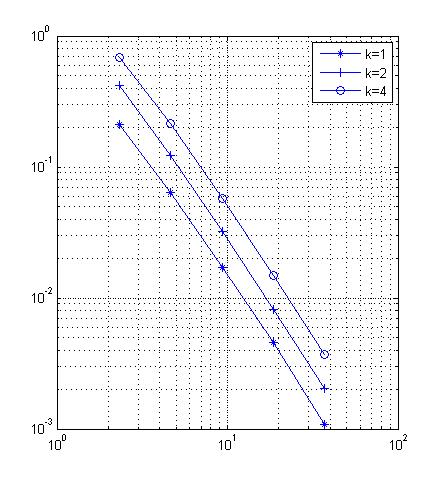} &
\includegraphics[scale=0.4]{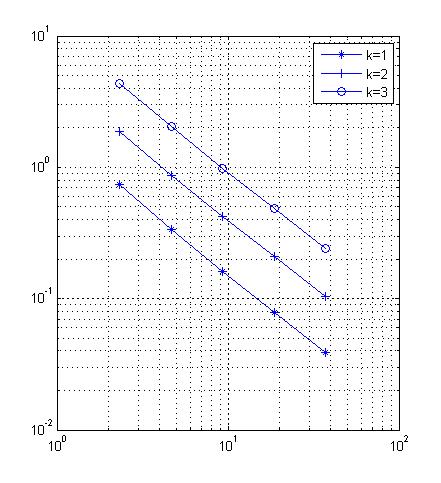} \\
\end{tabular}
\caption{Log-log plots for numerical errors (vertical) of ${U}$   vs. $1/h$ (horizontal). Left: ${\mathcal H}^0$-norm; right: ${\mathcal H}^1$-norm.}
\end{figure}

\begin{figure}[ht]
\centering
\includegraphics[scale=0.4]{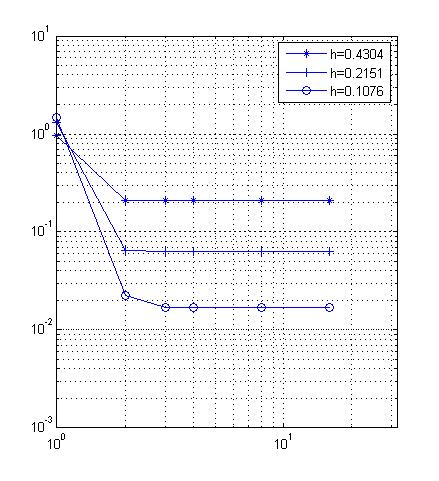}
\caption{Log-log plot for numerical errors $\|{U}-{U}_{h}\|_ {\mathcal H^0}$ (vertical) vs. $N$ (horizontal).}
\end{figure}



\appendix
\section{Proof of Theorem \ref{wellposedness2}}
\label{appendix1}
\renewcommand{\theequation}{A.\arabic{equation}}

Before proving Theorem \ref{wellposedness2}, we present some preliminary results.
\begin{definition}
\label{DefinitionA1}
Let $\langle\cdot,\cdot\rangle_1$ and $\langle\cdot,\cdot\rangle_2$ be inner products on $(H^1(\Om))^2\times (H^1(\Om))^2$ and $H^1(\Om_R)\times H^1(\Om_R)$ defined by, for $\forall\; u,v\in (H^1(\Om))^2$ and $\forall\; p,q\in H^1(\Om_R)$,
\ben
\label{innerproduct1}
\langle u,v\rangle_1 &=& \int_{\Om}{ \left[\lambda(\nabla\cdot u)(\nabla\cdot\ov{v})+\frac{\mu}{2}\left(\nabla{u}+(\nabla{ u})^T\right):\left(\nabla\ov{v}+(\nabla\ov{v})^T\right)+\widetilde{\alpha}{u} \cdot \ov {v}\right] dx},\\
\label{innerproduct2}
\langle p,q\rangle_2 &=& \int_{\Om_R}{\left(\nabla p\cdot \nabla \ov{q}+p \ov{q}\right) dx}+b_1(p,q),
\enn
which further induce the norms $|||\cdot|||_1$ on $(H^1(\Om))^2$ and $|||\cdot|||_2$ on $H^1(\Om_R)$, respectively, i.e., $\forall\,u\in (H^1(\Om))^2, p\in H^1(\Om_R)$,
\ben
|||u|||_1^2=\langle{u},{u}\rangle_1, \quad|||p|||_2^2=\langle p,p\rangle_2.
\enn
\end{definition}

From the above definition together with (\ref{a1vv}) and (\ref{b1}), we conclude that there exist constants $\alpha_0>0,\beta_0>0,\ga_0>0$ such that
\be
\label{equivilent1}
\alpha_0\|u\|_{\left(H^1(\Om)\right)^2}^2 \le |||u|||_1^2 \le \beta_0\|u\|_{\left(H^1(\Om)\right)^2}^2,\\
\label{equivilent2}
\|p\|_{H^1(\Om_R)}^2\le |||p|||_2^2 \le \ga_0\|p\|_{H^1(\Om_R)}^2.
\en
Now, let $\{\widetilde{\Phi}_i,\widetilde{\lambda}_i\}$ and $\{\widehat{\varphi}_i,\widehat{\lambda}_i\}$ be eigenpairs satisfying
\be
\label{eigen1}
\langle\widetilde{\Phi}_i,\Theta\rangle_1  &=& \widetilde{\lambda}_i
\left(\widetilde{\Phi}_i,\Theta\right)_{\left(L^2(\Omega)\right)^2}  , \quad\forall\; \Theta\in \left(H^1(\Omega)\right)^2,\\
\label{eigen2}
\langle\widehat{\varphi_i},\theta\rangle_2 &=& \widehat{\lambda}_i\left(\widehat{\varphi}_i,\theta\right)_{L^2(\Om_R)}, \quad\forall\; \theta\in H^1(\Om_R),
\en
where $\left(\cdot,\cdot\right)_{H}$ is the classical $L^2$ inner product on $H$. Without loss of generalities, we assume that $0<\widetilde{\lambda}_1\le\widetilde{\lambda}_2\le \ldots$, $0<\widehat{\lambda}_1\le\widehat{\lambda}_2\le \ldots$, and $(\widetilde\Phi_i,\widetilde\Phi_j)_{(L^2(\Omega))^2} =(\widehat\varphi_i,\widehat\varphi_j)_{L^2(\Omega_R)} =\delta_{ij}$.

\begin{lemma}
\label{TheoremA2}
Let $u=\sum_{n=1}^{\infty}{\widetilde c_n\widetilde{\Phi}_n}$, $p=\sum_{n=1}^{\infty}{\widehat{c}_n\widehat{\varphi}_n}$ and $U=(u,p)$.\\
(1). If $U\in \mathcal{H}^0$,
\be
\label{L4.1-1}
\|U\|_{\mathcal{H}^0}^2=\sum_{n=1}^{\infty}{(|\widetilde{c}_n|^2+|\widehat{c}_n|^2)}.
\en
(2). If $U\in \mathcal{H}^1$,
\be
\label{L4.1-2}
|||u|||_1^2 =\sum_{n=1}^{\infty}{\widetilde{\lambda}_i|\widetilde{c}_n|^2},\quad |||p|||_2^2=\sum_{n=1}^{\infty}{\widehat{\lambda}_i|\widehat{c}_n|^2}.
\en
(3). Let $H_{M_1}=\underset{\widetilde\lambda_i\le M_1}{span}\{\widetilde\Phi_i\},H_{M_2}=\underset{\widehat\lambda_i\le M_2}{span}\{\widehat\varphi_i\}$ and we define
\ben
H_{M_1}^{\perp}&=&\{{v}\in (H^1(\Om))^2:\langle{v},\Theta\rangle_1=0, \quad\forall\; \Theta\in H_{M_1}\},\\
H_{M_2}^{\perp}&=&\{q\in H^1(\Om_R): \langle q,\theta\rangle_1=0, \quad\forall\; \theta\in H_{M_2}\}.
\enn
Then we have
\be
\label{L4.1-3}
|||u|||_1^2\le M_1\|u\|_{(H^0(\Om))^2}^2,\quad|||p|||_2^2\le M_2\|p\|_{H^0(\Om_R)}^2,\quad\forall\; U\in H_{M_1}\times H_{M_2},
\en
and
\be
\label{L4.1-4}
\|u\|_{(H^0(\Om))^2}^2\le\frac{1}{M_1}|||u|||_1^2,\quad \|p\|_{H^0(\Om_R)}^2\le \frac{1}{M_2}|||p|||_2^2,\quad\forall\; U\in H_{M_1}^{\perp}\times H_{M_2}^{\perp}.
\en
\end{lemma}

\begin{lemma}
\label{TheoremA3}
Suppose $U=(u,p)\in \mathcal{H}^1$ satisfying
\ben
\label{L4.2-1}
\langle{u},\Theta\rangle_1&=&(\widetilde{f}_1,\Theta)_{(L^2(\Om))^2},\quad\forall\; \Theta\in (H^1(\Om))^2,\\
\label{L4.2-2}
\langle p,\theta\rangle_2&=&(\widehat f_2,\theta)_{L^2(\Om_R)},\quad\forall\; \theta\in H^1(\Om_R)
\enn
with $(\widetilde{f}_1,\widehat f_2)\in \mathcal{H}^0$. Then there exists a positive constant $c$ such that
\be
\label{L4.2-3}
\|u\|_{(H^2(\Om))^2}\le c\|\widetilde{f}_1\|_{(H^0(\Om))^2},\quad
\|p\|_{H^2(\Om_R)}\le c\|\widehat f_2\|_{H^0(\Om_R)}.
\en
\end{lemma}
\begin{proof}
The first assertion for $u$ follows from the classical regularity estimates (\cite{E98}). We now prove the regularity results for $p$. It follows that $b_1(\cdot,\cdot)$ is the corresponding sesquilinear form of the DtN map for the solution of homogeneous Laplace equation outside $\Om$. Now we define
\ben
\widetilde{p}(x)=
\begin{cases}
p(x),\quad &|x|\le R ,\cr \sum_{n\ge 0}\left(\frac{|x|}{R}\right)^n(a_n\cos n\theta+b_n\sin n\theta),\quad &|x|>R,\nonumber
\end{cases}
\enn
where $p(R,\theta)=\sum_{n\ge 0}(a_n\cos n\theta+b_n\sin n\theta)$. Note that for $\varphi\in C_0^\infty(\R^2)$,
\ben
\int_{\R^2\backslash\ov{\Omega}}\nabla\widetilde{p}\cdot\nabla\ov{\varphi}dx +\int_{\Om_R}\widetilde{p}\,\ov{\varphi}\,dx
&=& \langle \widetilde{p},\varphi\rangle_2-\int_{|x|\ge R}\Delta\widetilde{p}\,\ov{\varphi}\,dx\\
&=& \langle p,\varphi\rangle_2\\
&=& (\widehat f_2,\varphi).
\enn
Then we conclude from interior regularity estimates that $\widetilde{p}=p\in H^2(\Om_R)$ and
\ben
\|p\|_{H^2(\Om_R)} &\le& c\{\|\widetilde{p}\|_{H^1(\Om_{R+1})}+\|\widehat{f}_2\|_{H^0(\Om_R)}\}\\
&\le& c\|\widehat{f}_2\|_{H^0(\Om_R)}.
\enn
\end{proof}

\begin{corollary}
\label{TheoremA4}
If $U=(u,p)\in H_{M_1}\times H_{M_2}$,
\be
\label{C4.1-1}
\|u\|_{(H^2(\Om))^2}\le c M_1^{1/2}|||u|||_1,\quad \|p\|_{H^2(\Om_R)}\le cM_2^{1/2}|||p|||_2.
\en
\end{corollary}
\begin{proof}
For $U=(u,p)\in H_{M_1}\times H_{M_2}$, we can see that
\ben
u=\sum_{\widetilde\lambda_n\le M_1}{\widetilde{c}_n\widetilde{\Phi}_n}, \quad p=\sum_{\widehat\lambda_n\le M_2}{\widehat{c}_n\widehat{\varphi}_n}.
\enn
Then (\ref{eigen1}) and (\ref{eigen2}) imply that
\ben
\langle{u},\Theta\rangle_1&=&\left(\sum_{\widetilde\lambda_n\le M_1}{\widetilde\lambda_n\widetilde{c}_n\widetilde{\Phi}_n},\Theta\right),\quad\forall\; \Theta\in (H^1(\Om))^2,\\
\langle p,\theta\rangle_2&=&\left(\sum_{\widehat\lambda_n\le M_2}{\widehat\lambda_n\widehat{c}_n\widehat{\varphi}_n},\theta\right),\quad\forall\; \theta\in H^1(\Om_R).
\enn
Therefore, (\ref{C4.1-1}) follows from Lemma \ref{TheoremA3} and (\ref{L4.1-2}).
\end{proof}

\begin{corollary}
\label{TheoremA5}
If $U=(u,p)\in H_{M_1}^{\perp}\times H_{M_2}^{\perp}$, there exist positive constants $M_0,C_1$ such that for $M_1,M_2\ge M_0$,
\be
\label{C4.2-1}
\|U\|_{\mathcal{H}^1}^2\le C_1 \mbox{Re}\{A^N(U,U)\}.
\en
\end{corollary}
\begin{proof}
From (\ref{equivilent1}) and (\ref{equivilent2}), we know
\ben
\label{C4.2-2}
\alpha_0\|U\|_{\mathcal{H}^1}^2 \le |||{u}|||_1^2+\alpha_0|||p|||_2^2 \le \max\{\alpha_0,1\}(|||{u}|||_1^2 +|||p|||_2^2).
\enn
According to Definition \ref{DefinitionA1} and $b_1^N(p, p)\ge0$, we have
\be
\label{C4.2-3}
\alpha_0\|U\|_{\mathcal{H}^1}^2 &\le&
\max\{\alpha_0,1\}\mbox{Re}\{A^N({U},{U})\} +\max\{\alpha_0,1\}(2\mu+\rho\om^2)\|{u}\|_{(H^0(\Om))^2}^2\nonumber\\
&+&\max\{\alpha_0,1\}(k^2+1)\|p\|_{H^0(\Om_R)}^2
+ \max\{\alpha_0,1\}\mbox{Re}\{b_2^N(p,p)-a_3(u,p)-a_4(p,{u})\}.
\en
The Sobolev embedding theorem and the the interpolation inequality for $0<\theta=1/2+\varepsilon< 1$, $0<\varepsilon<1/2$ gives
\be
\label{C4.2-4}
|b_2^N(p,p)-a_3({u},p)-a_4(p,{u})|&\le& c(\|{u}\|_{(H^0(\G ))^2}^2+\|p\|_{H^0(\G_R)}^2)\nonumber\\
&\le& c(\|{ u}\|_{(H^{1/2+\varepsilon}(\Om))^2}^2+\|p\|_{H^{1/2+\varepsilon}(\Om_R )}^2)\nonumber\nonumber\\
&\le& c(\|{u}\|_{(H^0(\Om))^2}^{1-2\varepsilon}\|{u}\|_{(H^1(\Om ))^2}^{1+2\varepsilon} +\|p\|_{H^0(\Om_R )}^{1-2\varepsilon}\|p\|_{H^1(\Om_R )}^{1+2\varepsilon}).
\en
A combination of (\ref{equivilent1}), (\ref{equivilent2}) and (\ref{L4.1-4}) leads to
\be
\label{C4.2-5}
\|{u}\|_{(H^0(\Omega ))^2}^2\le\frac{\beta_0}{M_1}\|{u}\|_{(H^1(\Omega ))^2}^2,\|p\|_{H^0(\Omega_R )}^2\le\frac{\gamma_0}{M_2}\|p\|_{H^1(\Omega_R )}^2.
\en
Thus, we conclude from (\ref{C4.2-3})-(\ref{C4.2-5}) that
\be
\label{C4.2-6}
\|{U}\|_{\mathcal{H}^1}^2 &\le&
C_\alpha \mbox{Re}\{A^N({U},{U})\}+C_\alpha(2\mu+\rho\om^2)\|{ u}\|_{(H^0(\Om))^2}^2+C_\alpha(k^2+1)\|p\|_{H^0(\Om_R)}^2 \nonumber\\
&+& C_\alpha|b_2^N(p,p)-a_3({u},p)-a_4(p,{u})|\nonumber\\
&\le& C_\alpha \mbox{Re}\{A^N({U},{U})\}\nonumber\\
&+& C_\alpha\frac{\beta_0^2(2\mu+\rho\om^2)}{M_1^2} \|{u}\|_{(H^1(\Om ))^2}^2+ C_\alpha\frac{\ga_0^2(1+k^2)}{M_2^2}\|p\|_{H^1(\Om_R )}^2 \nonumber\\
&+& C_\alpha \frac{c\beta_0^{1/2-\varepsilon}}{M_1^{1/2-\varepsilon}}\|{ u}\|_{(H^1(\Om))^2}^2+C_\alpha \frac{c\ga_0^{1/2-\varepsilon}}{M_2^{1/2-\varepsilon}}\|p\|_{H^1(\Om_R)}^2,
\en
where $C_\alpha=\max\{1,1/\alpha_0\}$. Let $M_0>0$ be such that
\be
\label{C4.2-7}
1-C_\alpha\frac{\beta_0^2(2\mu+\rho\om^2)}{M_0^2}-C_\alpha \frac{c\beta_0^{1/2-\varepsilon}}{M_0^{1/2-\varepsilon}} >\frac{1}{2},\\
\label{C4.2-8}
1-C_\alpha\frac{\ga_0^2(1+k^2)}{M_0^2}-C_\alpha \frac{c\ga_0^{1/2-\varepsilon}}{M_0^{1/2-\varepsilon}} >\frac{1}{2}.
\en
Then for $M_1,M_2\ge M_0$, there exists a constant $C_1=2C_\alpha$ such that (\ref{C4.2-1}) holds.
\end{proof}

\begin{corollary}
\label{TheoremA6}
If ${U}=({u},p)\in H_{M_1}\times H_{M_2}$ and ${V}=(v,q)\in H_{M_1}^{\perp}\times H_{M_2}^{\perp}$, we have
\be
\label{C4.3-1}
|A({U},{V})|\le c\ga_0^{1/4-\varepsilon/2} M_2^{\varepsilon/2-1/4} \|p\|_{H^1(\Om_R)}\|q\|_{H^1(\Om_R)}.
\en
where $c>0$ and $0<\varepsilon<1/2$ are constants independent of ${U}$ and ${V}$.
\end{corollary}
\begin{proof}
It easily follows that
\ben
|A({U},{V})| &=& |\langle{u},{v}\rangle_1-(\widetilde{\alpha}+\rho\om^2)({u},{ v})_{(H^0(\Om))^2}+\langle p,q\rangle_2-(k^2+1)(p,q)_{H^0(\Om_R)}-b_2(p,q)|\\
&=& |b_2(p,q)|\\
&\le& c\|p\|_{H^{1/2+\varepsilon}(\Om_R)}\|q\|_{H^{1/2+\varepsilon}(\Om_R)},
\enn
where $c>0$ is a constant. Then the interpolation inequality and (\ref{C4.2-5}) leads to (\ref{C4.3-1}).
\end{proof}

\begin{corollary}
\label{TheoremA7}
If ${U}=({u},p)\in H_{M_1}\times H_{M_2}$, there exist ${V}=({v},q)\in H_{M_1}\times H_{M_2}$ and a positive constant $c$ such that
\be
\label{C4.4-1}
\|{U}\|_{\mathcal{H}^1}\le c\frac{\mbox{Re}\{A({U},{V})\}}{\|{ V}\|_{\mathcal{H}^1}}.
\en
\end{corollary}
\begin{proof}
Firstly, we observe that for ${W}=({w},\varpi)\in \mathcal{H}^1$, we can rewrite it as
\be
\label{C4.4-2}
{W}={W}_M+{W}_M^{\perp}=({w}_{M_1}+{ w}_{M_1}^{\perp},\varpi_{M_2}+\varpi_{M_2}^{\perp}),
\en
where ${w}_{M_1}\in H_{M_1}$, ${w}_{M_1}^{\perp}\in H_{M_1}^{\perp}$, $\varpi_{M_2}\in H_{M_2}$ and $\varpi_{M_2}^{\perp}\in H_{M_2}^{\perp}$. Then from (\ref{equivilent1}) and (\ref{equivilent2}) we know that
\be
\label{C4.4-3}
\|{W}_M\|_{\mathcal{H}^1}^2 &=& \|{w}_{M_1}\|_{(H^1(\Om ))^2}^2+\|\varpi_{M_2}\|_{H^1(\Om_R)}^2\nonumber\\
&\le& \frac{1}{\alpha_0}|||{w}_{M_1}|||_1 +|||\varpi_{M_2}|||_2\nonumber\\
&\le& \frac{1}{\alpha_0}|||{w}|||_1 +|||\varpi|||_2\nonumber\\
&\le& \frac{\beta_0}{\alpha_0}\|{ w}\|_{(H^1(\Om))^2}^2+\ga_0\|\varpi\|_{H^1(\Om_R)}^2 \nonumber\\
&\le& \max\left\{\frac{\beta_0}{\alpha_0},\ga_0\right\}\|{ W}\|_{\mathcal{H}^1}^2.
\en
Similarly, we have
\be
\label{C4.4-4}
\|{W}_M^{\perp}\|_{\mathcal{H}^1}^2 \le \max\left\{\frac{\beta_0}{\alpha_0},\ga_0\right\}\|{W}\|_{\mathcal{H}^1}^2.
\en
From the stability of (\ref{Weak1}), we know that there exists a ${W}=({ w},\varpi)\in \mathcal{H}^1$ such that
\ben
\|{U}\|_{\mathcal{H}^1}\le c\frac{\mbox{Re}\{A({U},{W})\}}{\|{ W}\|_{\mathcal{H}^1}},
\enn
where $c>0$ is a constant. Let ${V}={W}_M$. Therefore, Corollary \ref{TheoremA6}, equations (\ref{C4.4-3}) and (\ref{C4.4-4}) yield
\ben
\label{C4.4-5}
\|{U}\|_{\mathcal{H}^1} &\le& c\frac{\mbox{Re}\{A({U},{W})\}}{\|{ W}\|_{\mathcal{H}^1}} \nonumber\\
&=& c\frac{\mbox{Re}\{A({U},{V})+A({U},{W}-{V})\}}{\|{W}\|_{\mathcal{H}^1}} \nonumber\\
&\le& c \max\left\{\frac{\beta_0}{\alpha_0},\ga_0\right\}\frac{\mbox{Re}\{A({ U},{V})\}}{\|{V}\|_{\mathcal{H}^1}}+ c\gamma_0^{1/4-\varepsilon/2} M_2^{\varepsilon/2-1/4}\max\left\{\frac{\beta_0}{\alpha_0},\ga_0\right\} \|{U}\|_{\mathcal{H}^1}.
\enn
Let $M_0>0$ be such that
\be
\label{C4.4-6}
1-c\gamma_0^{1/4-\varepsilon/2} M_0^{\varepsilon/2-1/4}\max\left\{\frac{\beta_0}{\alpha_0},\ga_0\right\}>\frac{1}{2}.
\en
Thus, (\ref{C4.4-1}) holds for $M_2\ge M_0$.
\end{proof}

{\bf Proof of Theorem \ref{wellposedness2}.} Here we prove that for all ${U}=({u},p)\in \mathcal{H}^1$, there exists a constant $N_{0}=N_0\geq 0$ such that
\be
\label{T4.3-1}
\|{U}\|_{\mathcal{H}^1} \le c \sup_{({\bf 0},0)\neq {V}\in \mathcal{H}_{h}}\frac{|A^N({U},{V})|}{\|{V}\|_{\mathcal{H}^1}},
\en
where $c>0$ is a constant. Then the uniqueness follows immediately and then Theorem \ref{wellposedness2} follows from  the Fredholm Alternative theorem. We use the same form as (\ref{C4.4-2}) to represent ${U}$ as
\ben
\label{T4.3-2}
{U}={U}_M+{U}_M^{\perp}=({u}_{M_1}+{ u}_{M_1}^{\perp},p_{M_2}+p_{M_2}^{\perp}),
\enn
where ${u}_{M_1}\in H_{M_1}$, ${u}_{M_1}^{\perp}\in H_{M_1}^{\perp}$, $p_{M_2}\in H_{M_2}$ and $p_{M_2}^{\perp}\in H_{M_2}^{\perp}$.

If ${U}=0$, then (\ref{T4.3-1}) holds trivially.

If ${U}\ne0$, we can obtain from Corollary \ref{TheoremA7} that there exist ${ V}_M=({v}_{M_1},q_{M_2})\in H_{M_1}\times H_{M_2}$ and a positive constant $c$ such that
\be
\label{T4.3-3}
\|{U}_M\|_{\mathcal{H}^1}\le c\frac{\mbox{Re}\{A({U}_M,{V}_M)\}}{\|{\bf V}_M\|_{\mathcal{H}^1}}
\en
with $\|{U}_M\|_{\mathcal{H}^1}=\|{V}_M\|_{\mathcal{H}^1}$ after scaling. Now we define ${V}={V}_M+{U}_M^{\perp}$. Then (\ref{T4.3-3}) together with Corollary \ref{TheoremA5} give
\be
\label{T4.3-4}
\|{U}_M\|_{\mathcal{H}^1}^2+\|{U}_M^{\perp}\|_{\mathcal{H}^1}^2 &\le& c(\mbox{Re}\{A({U}_M,{V}_M)\}+\mbox{Re}\{A^N({U}_M^{\perp},{U}_M^{\perp})\}) \nonumber\\
&=& c|A^N({U},{V})|+c|A^N({U}_M,{U}_M^{\perp}|+c|A^N({U}_M^{\perp},{V}_M)| \nonumber\\
&+& c|b(p_{M_2},q_{M_2})-b^N(p_{M_2},q_{M_2})|.
\en
According to Theorem \ref{Theorem4.1}, we obtain
\be
\label{T4.3-5}
|b(p_{M_2},q_{M_2})-b^N(p_{M_2},q_{M_2})| &\le& \|(S-S^N)p_{M_2}\|_{H^{-3/2}(\G_R)}\|q_{M_2}\|_{H^{3/2}(\G_R)}  \nonumber\\
&\le& cN^{-2}\|p_{M_2}\|_{H^{3/2}(\G_R)} \|q_{M_2}\|_{H^{3/2}(\G_R)}\nonumber\\
&\le& cN^{-2}\|p_{M_2}\|_{H^2(\Om_R)}\|q_{M_2}\|_{H^2(\Om_R)} \nonumber\\
&\le& c_1N^{-2}M_2\gamma_0\|{U}_M\|_{\mathcal{H}^1}^2.
\en
Additionally, we have that
\be
\label{T4.3-6}
|A^N({U}_M,{U}_M^{\perp})|&\le& |A({U}_M,{U}_M^{\perp})|+
|b(p_{M_2},p_{M_2}^{\perp})-b^N(p_{M_2},q_{M_2}^{\perp})|.
\en
Then Theorem \ref{Theorem4.1} gives
\be
\label{T4.3-7}
|b(p_{M_2},p_{M_2}^{\perp})-b^N(p_{M_2},q_{M_2}^{\perp})| &\le& \|(S-S^N)p_{M_2}\|_{H^{-1/2}(\G_R)}\|q_{M_2}^{\perp}\|_{H^{1/2}(\G_R)}  \nonumber\\
&\le& cN^{-1}\|p_{M_2}\|_{H^{3/2}(\G_R)} \|q_{M_2}^{\perp}\|_{H^{1/2}(\G_R)}\nonumber\\
&\le& cN^{-1}\|p_{M_2}\|_{H^2(\Om_R)}\|q_{M_2}^{\perp}\|_{H^1(\Om_R)} \nonumber\\
&\le& c_2N^{-1}M_2^{1/2}\gamma_0^{1/2}\|{U}_M\|_{\mathcal{H}^1} \|{ U}_M^{\perp}\|_{\mathcal{H}^1}.
\en
Therefore, Corollary, (\ref{T4.3-6}), (\ref{T4.3-7}) and the arithmetic-geometric mean inequality lead to
\be
\label{T4.3-8}
|A^N({U}_M,{U}_M^{\perp})|&\le& (c_3\gamma_0^{1/4-\varepsilon/2} M_2^{\varepsilon/2-1/4}+c_4N^{-1}M_2^{1/2}\gamma_0^{1/2})(\|{ U}_M\|_{\mathcal{H}^1}^2+\|{U}_M^{\perp}\|_{\mathcal{H}^1}^2).
\en
Similarly, we have
\be
\label{T4.3-9}
|A^N({U}_M^{\perp},{V}_M)|&\le& (c_5\gamma_0^{1/4-\varepsilon/2} M_2^{\varepsilon/2-1/4}+c_6N^{-1}M_2^{1/2}\gamma_0^{1/2})(\|{ U}_M\|_{\mathcal{H}^1}^2+\|{U}_M^{\perp}\|_{\mathcal{H}^1}^2).
\en
Here, we need that $M_1,M_2\ge M_0$ where $M_0>0$ satisfies (\ref{C4.2-7}), (\ref{C4.2-8}) and (\ref{C4.4-6}). In additionally, we suppose that $M_0>0$ and $N_0\ge0$ are large enough such that
\be
\label{T4.3-10}
1-c_1N_0^{-2}M_0\gamma_0-(c_3+c_5)\gamma_0^{1/4-\varepsilon/2} M_0^{\varepsilon/2-1/4}-(c_5+c_6)N_0^{-1}M_0^{1/2}\gamma_0^{1/2}>\frac{1}{2},
\en
which further implies that there is a positive constant $c$ such that
\be
\label{T4.3-11}
\|{U}_M\|_{\mathcal{H}^1}^2+\|{U}_M^{\perp}\|_{\mathcal{H}^1}^2 \le c|A^N({ U},{V})|.
\en
Now, since
\ben
\alpha_0\|{u}\|_{(H^1{\Om})^2}^2+\|p\|_{H^1(\Om_R)}^2 &\le& |||{ u}|||_1^2+|||p|||_2^2 \nonumber\\
&=& |||{u}_{M_1}|||_1^2+|||{ u}_{M_1}^{\perp}|||_1^2+|||p_{M_2}|||_2^2+|||p_{M_2}^{\perp}|||_2^2 \nonumber\\
&\le& \beta_0(\|{u}_{M_1}\|_{(H^1{\Om})^2}^2+\|{ u}_{M_1}^{\perp}\|_{(H^1{\Om})^2}^2)+ \gamma_0(\|p_{M_2}\|_{H^1(\Om_R)}^2+\|p_{M_2}^{\perp}\|_{H^1(\Om_R)}^2) \nonumber\\
&\le& \max\{\beta_0,\gamma_0\}(\|{U}_M\|_{\mathcal{H}^1}^2+\|{ U}_M^{\perp}\|_{\mathcal{H}^1}^2),
\enn
we have
\be
\label{T4.3-12}
\|{U}\|_{\mathcal{H}^1}^2 \le \frac{\max\{\beta_0,\gamma_0\}}{\min\{1,\alpha_0\}}(\|{ U}_M\|_{\mathcal{H}^1}^2+\|{U}_M^{\perp}\|_{\mathcal{H}^1}^2).
\en
Thus,
\be
\label{T4.3-13}
\|{U}\|_{\mathcal{H}^1}^2 \le c|A^N({U},{V})|
\en
for some constant $c>0$.
Finally, we obtain from (\ref{T4.3-12}) similarly that
\be
\label{T4.3-14}
\|{V}\|_{\mathcal{H}^1} &\le& \sqrt{\frac{\max\{\beta_0,\gamma_0\}}{\min\{1,\alpha_0\}}\left(\|{ V}_M\|_{\mathcal{H}^1}^2+\|{ U}_M^{\perp}\|_{\mathcal{H}^1}^2\right)}\nonumber\\
&=& \sqrt{\frac{\max\{\beta_0,\gamma_0\}}{\min\{1,\alpha_0\}}\left(\|{ U}_M\|_{\mathcal{H}^1}^2+\|{ U}_M^{\perp}\|_{\mathcal{H}^1}^2\right)}\nonumber\\
&\le& \sqrt{\frac{\max\{\beta_0,\gamma_0\}}{\min\{1,\alpha_0\}} \left(\frac{1}{\alpha_0}|||{u}|||_1^2+|||p|||_2^2\right)}\nonumber\\
&\le& \sqrt{\frac{\max\{\beta_0,\gamma_0\}}{\min\{1,\alpha_0\}} \max\left\{\frac{\beta_0}{\alpha_0},\gamma_0\right\}}\|{U}\|_{\mathcal{H}^1}.
\en
Therefore, (\ref{T4.3-13}) and (\ref{T4.3-14}) give
\ben
\|{U}\|_{\mathcal{H}^1} &\le& c\frac{|A^N({U},{V})|}{\|{V}\|_{\mathcal{H}^1}} \nonumber\\
&\le& c \sup_{(0,0)\neq {V}\in \mathcal{H}_{h}}\frac{|A^N({U},{V})|}{\|{ V}\|_{\mathcal{H}^1}},
\enn
where $c>0$ is a constant. This completes the proof.

\section*{Acknowledgments}
L. Xu is partially supported by a Key Project
of the Major Research Plan of NSFC (No. 91630205), and a NSFC Grant (11771068), and he also would like to thank  Prof. G.C. Hsiao and Prof. J.E. Pasciak  for their invaluable encouragements and suggestions  which are of great  importance for the completion of this work.

\end{document}